\documentclass[11pt,psamsfonts,oneside]{amsart}
\usepackage{latexsym}
\usepackage{amsmath, amsfonts, amssymb, amscd}
\usepackage[all]{xy}
\usepackage[height=22.5cm, width=16cm]{geometry}
\usepackage[latin1]{inputenc}
\usepackage[colorlinks=true, pdfstartview=FitV, linkcolor=blue, citecolor=blue, urlcolor=blue]{hyperref}

\newtheorem{definition}{Definition}[section]
\newtheorem{theorem}[definition]{Theorem}
\newtheorem{lemma}[definition]{Lemma}
\newtheorem{corollary}[definition]{Corollary}
\newtheorem{proposition}[definition]{Proposition}
\newtheorem*{Rtheorem}{Restriction Theorem}

\theoremstyle{definition}
\newtheorem{example}[definition]{Example}
\theoremstyle{remark}
\newtheorem*{remark}{Remark}

\DeclareMathOperator{\Gl}{GL} \DeclareMathOperator{\So}{SO}
\DeclareMathOperator{\Sl}{SL} 
\DeclareMathOperator{\Ad}{Ad} 
\DeclareMathOperator{\grad}{grad}

\newcommand{\im}{\mathrm{i}}
\newcommand{\ce}{\mathbb{C}}
\newcommand{\er}{\mathbb{R}}

\newcommand{\quot}[2]{#1/\!\!/#2}
\newcommand{\Mp}{\mathcal M_\liep}
\newcommand{\Mnp}{\mathcal M_{\np}}
\newcommand{\nor}{\mathcal N}
\newcommand{\zet}{\mathcal Z}
\newcommand{\xh}{X^{<H>}}
\newcommand{\xha}{\overline{X^{<H>}}}

\newcommand{\lieg}{\mathfrak{g}}
\newcommand{\lieh}{\mathfrak{h}}
\newcommand{\liep}{\mathfrak{p}}

\newcommand{\lieu}{\mathfrak{u}}
\newcommand{\lien}{\mathfrak{n}}
\newcommand{\lieq}{\mathfrak{q}}

\newcommand{\np}{\lien_\liep}

\newcommand{\ddt}{\left.\frac{\partial}{\partial
t}\right|_0}

\title[A Quotient Restriction Theorem]{A Quotient Restriction Theorem for actions of real reductive groups}
\author{Henrik St\"otzel}

\address{Fakult\"at f\"ur Mathematik\\
  Ruhr Universit\"at Bochum\\
  Universit\"atsstrasse 150\\
  D - 44780 Bochum}
\email{henrik.stoetzel@rub.de}
\thanks{The author is supported by the
  Sonderforschungsbereich SFB/TR12 of the Deutsche
  Forschungsgemeinschaft}

\begin{document}

\begin{abstract}
We prove a version of the Chevalley Restriction Theorem for the
action of a real reductive group $G$ on a topological space $X$
which locally embeds into a holomorphic representation. Assuming
that there exists an appropriate quotient $\quot XG$ for the
$G$-action, we introduce a stratification which is defined with
respect to orbit types of closed orbits. Our main result is a
description of the quotient $\quot XG$ in terms of quotients by
normalizer subgroups associated to the stratification.
\end{abstract}

\maketitle

\section{Introduction}

Let $G$ be a connected complex semisimple Lie group and let us
consider the adjoint representation of $G$ on its Lie algebra
$\lieg$. If $\lieh$ is a Cartan subalgebra of $\lieg$, then the
inclusion $\lieh\hookrightarrow \lieg$ induces a homomorphism of
the algebra $\ce[\lieg]$ of polynomials on $\lieg$ into the
algebra $\ce[\lieh]$ of polynomials on $\lieh$. Let $\ce[\lieg]^G$
denote the set of invariant polynomials on $\lieg$ with respect to
the adjoint action of $G$ on $\lieg$. If $\mathcal W$ is the Weyl
group of $\lieg$ with respect to $\lieh$, we denote by
$\ce[\lieh]^{\mathcal W}$ the set of $\mathcal W$-invariant
polynomials on $\lieh$. The Chevalley restriction theorem states
\medskip

(Chevalley) \textit{The inclusion $\lieh\hookrightarrow\lieg$
induces an isomorphism $\ce[\lieg]^G\to \ce[\lieh]^{\mathcal W}$.}

\medskip

In terms of algebraic quotients this means that we have a
canonical isomorphism $\quot{\lieh}{\mathcal W}\to\quot{\lieg}{G}$
of algebraic varieties.

The Chevalley restriction theorem was generalized by Luna and
Richardson (\cite{LunaRich}) and by Schwarz (\cite{Schwarz80}) to
actions of complex reductive groups on affine varieties and to
actions of compact groups on smooth manifolds, respectively.

Assume that $X$ is an irreducible normal affine variety equipped
with a regular action of a complex reductive group and denote by
$\pi\colon X\to\quot XG$ the algebraic quotient. Then there exists
a Zariski open subset $U$ of $X$ and a subgroup $H$ of $G$ such
that each fiber of the restricted quotient $U\to\quot UG$ contains
a closed orbit of orbit type $G/H$. Let $X^H$ be the set of
$H$-fixed points in $X$ and let $\nor_G(H)$ be the normalizer of
$H$ in $G$.

\medskip

(Luna, Richardson) \textit{Assume that the quotient
$\quot{X^H}{\nor_G(H)}$ is irreducible. Then the inclusion
$X^H\hookrightarrow X$ induces an isomorphism
$\quot{X^H}{\nor_G(H)}\to\quot{X}{G}$ of affine varieties.}

\medskip

Let $G$ be compact and let $X$ be a smooth $G$-manifold such that
the quotient $X/G$ is connected. Then there exists a generic
isotropy group, i.\,e.\ a subgroup $H$ of $G$ such that in a
$G$-invariant dense open subset of $X$ each orbit is of orbit type
$G/H$. Let $\xh:=\{x\in X; G_x=H\}$ and let $\xha$ denote the
closure of $\xh$.

\medskip

(Schwarz) \textit{The inclusion $\overline{X^{<H>}}\hookrightarrow
X$ induces a homeomorphism $\overline{X^{<H>}}/\nor_G(H)\to X/G$.}

\medskip

We give a version of the Chevalley Restriction Theorem for actions
of real reductive groups on topological spaces which are locally
embedded into representations.

More precisely our setup is as follows. Let $U^\ce$ be a complex
reductive group with compact real form $U$. Then the map $U\times
\im\lieu\to U^\ce$, $(u,\xi)\mapsto u\exp(\xi)$, is a
diffeomorphism. We call a closed subgroup $G$ of $U^\ce$ real
reductive, if $G=K\exp(\liep)$ for $K:=G\cap U$ and
$\liep:=\lieg\cap\im\lieu$.

If $X$ is a $G$-representation which is given by restriction of a
finite dimensional holomorphic $U^\ce$-representation, then the
topological Hilbert quotient $\pi\colon X\to\quot XG$ exists. By
definition, this quotient identifies two points in $X$ if and only
if the closures of the $G$-orbits through these points intersect.
For a subset $Y\subset X$ which is $G$-saturated, i.\,e.\ for
which $\pi^{-1}(\pi(Y))=Y$ holds, or which is closed and
$G$-invariant, the quotient $Y\to \quot YG$ exists and is given by
the restriction $\pi|Y\colon Y\to\pi(Y)$. We call a topological
$G$-space $X$ a locally $G$-semistable space, if the topological
Hilbert quotient $\pi\colon X\to\quot XG$ exists, and if for each
$x\in X$ there exists a $G$-saturated open neighborhood $W$ of $x$
and a complex reductive group $U^\ce$ containing $G$ as a closed
compatible subgroup such that $W$ is $G$-equivariantly
homeomorphic to a closed $G$-invariant subset of an open
$G$-saturated subset of a holomorphic $U^\ce$-representation space
$V$.

If $X$ is a locally $G$-semistable space and $G\cdot x$ is a
closed orbit in $X$, then there exists a geometric $G$-slice at
$x$ as follows. Let $G_x$ be the $G$-isotropy group at $x$ and let
$S$ be a $G_x$-invariant locally closed subset of $X$ which
contains $x$. Then $G_x$ acts on $G\times S$ by $h\cdot
(g,s)=(gh^{-1},h\cdot s)$ and the restricted action $G\times S\to
X$, $(g,s)\mapsto g\cdot s$ induces a map $\Psi\colon
G\times^{G_x}S\to GS$ where $G\times^{G_x}S$ denotes the quotient
$(G\times S)/G_x$. We call $S$ a geometric $G$-slice, if $\Psi$ is
a homeomorphism onto an open subset. Here it follows from the
construction that the slice $S$ can be chosen such that $GS$ is
$G$-saturated and $S$ is a locally $G_x$-semistable space.

For a locally $G$-semistable space $X$, each fiber of the quotient
$\pi\colon X\to\quot XG$ contains a unique closed orbit and this
orbit is also the unique orbit of minimal dimension in that fiber.
For a compatible subgroup $H$ of $G$, we define $\xh$ to be the
set of points $x\in X$ such that $G\cdot x$ is closed and $G_x=H$.
We call $I_H(X):=\pi^{-1}(\pi(\xh))$ the $G$-isotropy stratum of
$H$ in $X$. One of our results is that the sets $I_H(X)$ define a
stratification of $X$.

If $G$ is complex reductive and acts regularly on an irreducible
complex space, then there exists a dense stratum. This is also the
case if $G$ is compact and $X$ is a smooth $G$-manifold such that
the quotient $X/G$ is connected. For actions of real reductive
groups, the stratification is more delicate. Even if $X$ is a
$G$-representation space, a dense stratum does not necessarily
exist. Originally we were interested in actions on smooth
manifolds but the fact that strata are not smooth in general
motivated our definition of a locally $G$-semistable space. Here a
stratum in $X$ and even the closure of a stratum are again locally
$G$-semistable spaces and they contain a dense stratum.

In the following, we assume that the locally $G$-semistable space
$X$ contains a dense stratum. Let $x_0\in X$ and let $G\cdot x$ be
the unique closed orbit in the fiber $\pi^{-1}(\pi(x_0))$. Let
$G\times^{G_x}S\to GS$ be a slice at $x$. Since $S$ is a locally
$G_x$-semistable space, we have the notion of $G_x$-isotropy
strata in $S$. We define $n(x_0)$ to be the number of open
$G_x$-isotropy strata in $S$ which contain $x$ in their closure
and we call $n(x_0)$ the splitting number at $x_0$. Note that the
splitting number is one in the simple cases where the existence of
a dense stratum is guaranteed.

Our main result is the following.

\begin{Rtheorem}
Let $X$ be a locally $G$-semistable space containing an open and
dense $G$-isotropy stratum $I_H(X)$. Then the topological Hilbert
quotient $\pi_N\colon\xha\to\quot{\xha}{\nor_G(H)}$ exists and the
inclusion $\xha\hookrightarrow X$ induces a continuous finite
surjective map
\[\Phi\colon\quot{\xha}{\nor_G(H)}\to \quot{X}{G}.\] For $x\in X$,
the number of points in the fiber $\Phi^{-1}(\pi(x))$ is equal to
the splitting number $n(x)$. If $x\in\xha$ with $n(x)>1$, then
$\Phi$ is not open at $\pi_N(x)$.
\end{Rtheorem}

Here by a finite map, we mean a proper map with finite fibers. In
the special case where the splitting number is constantly one, the
map $\Phi$ is a homeomorphism. The splitting number is one for
points in the stratum $I_H(X)$ and $I_H(X)$ is again a locally
$G$-semistable space, so in particular the restriction
$\quot{\xh}{\nor_G(H)}\to\quot{I_H(X)}G$ of $\Phi$ is a
homeomorphism.

Our result is new even for a representation of a semisimple real
group or more generally if $X$ is a $G$-representation space which
is given by restriction of a holomorphic $U^\ce$-representation.
If the representation space $X$ contains a dense stratum, then
$\xha$ is the subspace of $H$-fixed points in $X$. In particular
$\xha$ is smooth. This holds also if $X$ is a smooth locally
$G$-semistable space containing a dense stratum.

If $G$ is complex reductive and if $X$ is a holomorphic
$G$-representation, a dense stratum always exists and the
splitting number is constantly one. Consider the adjoint
representation of the complex reductive group $G$ on its Lie
algebra $X=\lieg$. Here $\xha$ is a Cartan subalgebra $\lieh$ of
$\lieg$. The group $H$ acts trivially on $\lieh$, so the quotient
$\quot{\xha}{\nor_G(H)}$ coincides with the geometric quotient
$\lieh/\mathcal W$ where $\mathcal W=\nor_G(H)/H$ is the Weyl
group. Thus we get the assertion of the Chevalley restriction
theorem. For the adjoint representation of a real reductive group,
a dense stratum does not necessarily exist. But for each open
stratum $I_H(\lieg)$, the Lie algebra $\lieh$ is a real Cartan
subalgebra. So, restricting the quotient $\quot \lieg G$ to
$I_H(\lieg)$ we get a version of the Chevalley restriction theorem
for the adjoint representation of a real reductive group.

If $X$ is an irreducible normal variety equipped with a regular
action of a complex reductive group $G$ then a dense stratum
exists and the splitting number is constantly one. Due to
normality of $X$, the map $\Phi$ is an isomorphism of varieties.
If the quotient $\quot{X^H}{\nor_G(H)}$ is irreducible then it
coincides with $\quot{\xha}{\nor_G(H)}$ and we obtain the theorem
of Luna and Richardson where it is assumed that
$\quot{X^H}{\nor_G(H)}$ is irreducible.

If $G$ is compact and $X$ is smooth, then the condition that $X/G$
is connected guarantees that there exists a dense stratum in $X$.
Moreover the splitting number is constantly one, so we obtain the
result of Schwarz.

The author would like to thank P.~Heinzner and G.~Schwarz for
helpful discussions and remarks on the content of this paper.

\section{Locally semistable spaces}

\subsection{Real reductive groups}

Let $U^\ce$ be a complex reductive Lie group with compact real
form $U$. If $\lieu$ denotes the
Lie algebra of $U$, then $U^\ce=U\exp(\im\lieu)$. Here the map
$U\times\im\lieu\to U^\ce$, $(u,\xi)\mapsto u\exp(\xi)$ is a
diffeomorphism. Note that $U^\ce$ is the universal
complexification of $U$ in the sense of \cite{Hochschild}.
We denote by $\theta$ the Cartan involution with fixed point set $U$.

We say that a Lie subgroup $G$ of $U^\ce$ is \emph{compatible}, if
$G=K\exp(\liep)$ for a subgroup $K$ of $U$ and a subspace $\liep$
of $\im\lieu$. A $\theta$-stable closed subgroup of $U^\ce$ is
compatible if and only if it has only finitely many connected
components (see e.g.~Lemma~1.1.3 in \cite{Mie07}). We call
$G=K\exp(\liep)$ the Cartan decomposition of $G$. Note that $K$ is
compact if and only if $G$ is closed. Moreover, we call a Lie
subgroup $H$ of $G$ \emph{compatible} if it is compatible with the
Cartan decomposition of $U^\ce$ or equivalently, if there exists a
subgroup $L$ of $K$ and a subspace $\lieq$ of $\liep$ such that
$H=L\exp(\lieq)$. Note that the latter condition depends only on
the Cartan decomposition of $G$ and not on the choice of $U^\ce$.
In the rest of this paper, $G$ will denote a closed compatible
subgroup of a fixed complex reductive group $U_G^\ce$ and
$G=K\exp(\liep)$ will denote the associated Cartan decomposition
of $G$.

\subsection{The topological Hilbert quotient}

Let $X$ be a topological $G$-space, i.\,e.\ a topological space
equipped with a continuous action of $G$. We define a relation on
$X$ by setting $x\sim y$ if and only if the closures
$\overline{G\cdot x}$ and $\overline{G\cdot y}$ of the orbits
$G\cdot x$ and $G\cdot y$ intersect. If this relation is an
equivalence relation, we define $\quot XG:=X/\sim$ and call the
quotient $\pi\colon X\to\quot XG$ the \emph{topological Hilbert
quotient}. If every $G$-orbit in $X$ is closed, in particular if
$G$ acts properly on $X$, then the quotient $\quot XG$ is the
usual orbit space $X/G$. This happens automatically if $G$ is
compact.

Assume that the topological Hilbert quotient $X\to\quot XG$
exists. A subset $Y\subset X$ is called \emph{$G$-saturated}, if
$Y=\pi^{-1}(\pi(Y))$ holds, or equivalently, if $x\in Y$ whenever
there exists a $y\in Y$ with $\overline{G\cdot
x}\cap\overline{G\cdot y}\neq\emptyset$. We say that a subset $Y$
of $X$ is \emph{$G$-open} if it is open and $G$-saturated.
Furthermore, we call a subset $Y$ of $X$ a \emph{$G$-locally
closed subset}, if the following equivalent conditions are
satisfied.
\begin{itemize}
\item $Y$ is a $G$-invariant closed subset of a $G$-open subset of
$X$. \item $Y$ is the intersection of a $G$-open and a closed
$G$-invariant subset of $X$. \item $Y$ is $G$-invariant, locally
closed, and an orbit $G\cdot y\subset Y$ is closed in $Y$ if and
only if it is closed in $X$.
\end{itemize}
For a $G$-locally closed subset $Y$ of $X$ the quotient
$Y\to\quot{Y}G$ exists and is obtained by restriction of the
quotient $X\to\quot XG$.

\subsection{Semistable spaces}

Recall that $G$ is a compatible subgroup of a complex reductive
group $(U_G)^\ce$ with Cartan involution $\theta$. Let $\rho\colon
G\to \Gl(V)$ be a representation of $G$ on a finite dimensional
complex vector space $V$ such that
$\rho\circ\theta=\theta'\circ\rho$ for a Cartan involution
$\theta'$ of $\Gl(V)$. Equivalently, we assume that we are given a
complex reductive group $U^\ce$ which contains $G$ as a compatible
subgroup and a holomorphic representation $\rho\colon
U^\ce\to\Gl(V)$. For this, note that $U^\ce:=(U_G)^\ce\times
\Gl(V)$ is complex reductive and that the map $g\mapsto
(g,\rho(g))$ embeds $G$ into $U^\ce$ and respects the Cartan
decomposition. Conversely, given a holomorphic representation of a
complex reductive group $U^\ce$, there exists a Cartan involution
$\theta'$ of $\Gl(V)$ which contains the compact group $\rho(U)$
in its fixed point set.

We may assume that $U$ acts by unitary operators on $V$ by
choosing a $U$-invariant hermitian inner product
$\langle\cdot,\cdot\rangle$ on $V$. Let $f\colon V\to \er$,
$f(v)=\frac 12\|v\|^2:=\frac 12\langle v,v\rangle$. Then $f$ is a
$U$-invariant strictly plurisubharmonic exhaustion function on $V$
and we get an induced $U$-invariant K\"{a}hler form
$\omega:=2\im\partial\overline{\partial}f$. The $U$-action on $V$
admits a moment map, i.\,e.\ a $U$-equivariant map $\mu\colon
V\to\lieu^*$ with $d\mu^\xi=\iota_{\xi_V}\omega$. Here $U$ acts on
$\lieu^*$ by the coadjoint action, $\mu^\xi$ is by definition the
function $\mu^\xi(v)=\mu(v)(\xi)$, the fundamental vector field
$\xi_V$ is given by $\xi_V(v)=\ddt \rho(\exp(t\xi)) v$ and
$\iota_{\xi_V}$ is the contraction of $\omega$ with $\xi_V$.
Explicitly, a moment map is given by
\[\mu^\xi(v):=\ddt f(\exp(\im t\xi)\cdot v)=\im \langle\xi_V(v),v\rangle, \]where we identify $T_vV$ with $V$.

Identifying $\lieu^*$ and $\im\lieu$ with respect to a
$U$-invariant inner product $\langle\cdot,\cdot\rangle$ on
$\lieu^\ce$, the moment map induces by restriction a map
$\mu_\liep\colon V\to\liep$, given by
$\langle\mu_\liep(v),\xi\rangle:=\mu_\liep^\xi(v):=\mu^{-\im\xi}(v)$.
Then $\mu_\liep$ is $K$-equivariant and satisfies the defining
equation $\grad \mu_\liep^\xi=\xi_V$ where the gradient is taken
with respect to the Riemannian metric on $V$ associated to the
K\"{a}hler metric. Therefore we call $\mu_\liep$ a \emph{$G$-gradient map}.

We call a $G$-invariant locally closed subset $X$ of $V$ a
\emph{$G$-semistable space} if each $G$-orbit which is closed in
$X$ is also closed in $V$. In the following we consider the
restriction $\mu_\liep\colon X\to\liep$ of the gradient map to
$X$. Define $\Mp$ to be the zero fiber $\mu_\liep^{-1}(0)\subset
X$.

The following results are established in \cite{RS} and
\cite{HSch}.

\begin{theorem}\label{theorem:HSch}
The topological Hilbert quotient $\pi\colon
X\to\quot{X}{G}$ exists and has the following properties.
\begin{enumerate}\item Each fiber of $\pi$ contains a unique
closed $G$-orbit.\label{thm:HSch,item:1}\item Each non-closed
orbit in a fiber of $\pi$ has strictly larger dimension than the
closed orbit in that fiber and contains the closed orbit in its
closure.\label{thm:HSch,item:2}\item For $x\in\Mp$ the orbit
$G\cdot x$ is the unique closed orbit in the fiber
$\pi^{-1}(\pi(x))$.\label{thm:HSch,item:3}\item The inclusion
$\Mp\hookrightarrow X$ induces a homeomorphism
$\Mp/K\to\quot{X}{G}$. \[ \xymatrix{
\Mp\ar[d] \ar@{^{(}->}[r]    & X\ar[d] \\
\Mp/K\ar@{->}[r]^{\sim} & \quot{X}{G}}\]
\label{thm:HSch,item:4}\item The restriction $\pi|\Mp$ of the
quotient map is proper and open.\label{thm:HSch,item:5}
\end{enumerate}
\end{theorem}

In particular, the theorem states that the orbits which intersect
$\Mp$ are exactly the closed orbits and that $\Mp$ intersects each
closed orbit in a $K$-orbit. So the quotient $\Mp/K$ parameterizes
the closed $G$-orbits in $X$.

\begin{remark}
Note that in particular the quotient $V\to\quot VG$ exists. Then
$X\subset V$ is a $G$-semistable space if and only if it is
$G$-locally closed in $V$.
\end{remark}

For later use we record

\begin{corollary}\label{corollary:MpSchnitt}
Let $Y\subset X$ be $G$-locally closed. Then
$\overline{Y\cap\Mp}=\overline Y\cap\Mp$.
\end{corollary}

\begin{proof}
Let $x\in\overline Y\cap\Mp$ and let $x_n\in Y$ be a sequence
converging to $x$. Then there exist $y_n\in\overline{G\cdot
x_n}\cap\Mp$. Note that $y_n\in Y$ since $Y$ is $G$-locally
closed. The sequence $\pi(y_n)=\pi(x_n)$ converges to $\pi(x)$.
Since the restriction of $\pi$ to $\Mp$ is proper, we may assume
that $y_n$ converges to a $y\in\Mp$. Then $y\in\overline Y$ and
$\pi(y)=\pi(x)$. Since $\pi^{-1}(\pi(x))\cap\Mp$ is a $K$-orbit we
conclude $x\in K\cdot y\subset \overline{Y\cap\Mp}$.
\end{proof}

\begin{corollary}\label{corollary:SaturatedNeighborhood}
Let $G\cdot x\subset X$ be a closed orbit. Then every
$G$-invariant open neighborhood of $G\cdot x$ contains a $G$-open
neighborhood.
\end{corollary}

\begin{proof}
Let $W$ be a $G$-invariant open neighborhood of $G\cdot x$. The set
$\pi^{-1}(\pi(W\cap\Mp))$ is $G$-open, contains $G\cdot x$ and is
contained in $W$.
\end{proof}

If $H$ is a closed compatible subgroup of $G$ then it follows from the
definition that a $G$-semistable space is also an $H$-semistable
space. Compatible subgroups of $G$ occur naturally as isotropy
groups of closed orbits.

\begin{lemma}\label{lemma:CompIsotropy}(Lemma~5.5 in \cite{HSch})
Let $X$ be a $G$-semistable space and let $x\in\Mp$. Then the
isotropy group $G_x$ is a closed compatible subgroup of $G$ with
Cartan decomposition $G_x=K_x\exp(\liep_x)$ where
$\liep_x:=\{\xi\in\liep;\,\xi_X(x)=0\}$.
\end{lemma}

\subsection{The Slice Theorem}

Let $X\subset V$ be a $G$-semistable space. From \cite{HSch} we
recall the construction of a slice at a point $x\in\Mp$ in the
case $X=V$. First, the action of $G$ on $X$ induces an isotropy
representation of $G_x$ on the tangent space $T_xX$. Then $T_xX$
can be $G_x$-equivariantly identified with $X$. Note that $T_xX$
is a $G_x$-semistable space since $G_x$ is a compatible subgroup
of $G$. The tangent space $T_x(G\cdot x)$ of the orbit $G\cdot x$
is a $G_x$-invariant subspace of $T_xX$.

\begin{lemma}[Corollary~14.9 in \cite{HSch}]\label{lemma:vollstReduzibel}
The $G$-representation $V$ is completely reducible.
\end{lemma}

Since $G_x$ is compatible, $V\cong T_xX$ is completely reducible
as a $G_x$-representation. Therefore there exists a
$G_x$-invariant subspace $W$ of $T_xX$ such that $T_xX=T_x(G\cdot
x)\oplus W$. The isotropy $G_x$ acts on the product $G\times W$ by
$h\cdot (g,s)=(gh^{-1},hs)$. We denote the quotient of this action
by $G\times^{G_x}W$ and we denote the element $G_x\cdot (g,s)$ by
$[g,s]$. The action of $G$ on $X$ induces a map $G\times^{G_x}W\to
GW\subset X$. There exists a $G_x$-invariant neighborhood $S$ of
$0$ in $W$ such that the restriction $G\times^{G_x}S\to GS$ is a
diffeomorphism onto an open subset $GS$ of $X$. Moreover, by
Corollary~\ref{corollary:SaturatedNeighborhood}, $S$ can be chosen
such that $GS$ is $G$-saturated in $X$ and $S$ is $G_x$-open in
$T_xX$. In particular, $S$ is a $G_x$-semistable space.

Similarly, for an arbitrary $G$-semistable space $X$ we get
(\cite{S08}, Corollary~4.5, Lemma~4.9)

\begin{theorem}[Slice Theorem] \label{slicetheorem}
Let $X$ be a $G$-semistable space and let $x\in \Mp$. Then there
exists a locally closed $G_x$-stable subset $S$ of $X$ containing
$x$, such that $GS$ is $G$-open in $X$ and such that the map
\[\Psi\colon G\times^{G_x}S\to G S,\quad
\Psi([g,s])=gs\] is a homeomorphism. The slice $S$ is
$G_x$-equivariantly homeomorphic to a $G_x$-locally closed subset
of $T_xV\cong V$ which contains $0$.
\end{theorem}

With the same notation as in the Slice Theorem, we call the data
$(G_x,S,V)$ a slice model at $x$. We will identify $S\subset X$
with the corresponding $G_x$-locally closed subset of $V$ without
further mentioning it. Let $(G_x,S,V)$ be a slice model at
$x\in\Mp$ and $g\in G$. Then $G_{gx}=gG_xg^{-1}$ and $g\cdot S$
contains $gx$. Then we get a homeomorphism $G\times^{G_{gx}}gS\to
GS$. This shows that the assumption $x\in\Mp$ could be replaced by
the assumption that $G\cdot x$ is closed. Note that $G_{gx}$ is
not necessarily a compatible subgroup of $G$.

\subsection{Locally semistable spaces}

Let $X$ be a topological
$G$-space such that the topological Hilbert quotient $\pi\colon
X\to\quot XG$ exists. We call $X$ a \emph{locally $G$-semistable
space} if every $x\in X$ has a $G$-open neighborhood $W$
which admits the structure of a $G$-semistable space, i.\,e.\
there exists a complex reductive group $U^\ce$ containing $G$ as a
compatible subgroup and a holomorphic $U^\ce$-representation space
such that $W$ is $G$-equivariantly homeomorphic to a
$G$-locally closed subset of $V$.

\begin{example}
Let $X$ be an affine complex variety equipped with a regular
action of a complex reductive group $G$ such that the algebraic
Hilbert quotient exists. Then the topological Hilbert quotient
exists and topologically coincides with the algebraic quotient.
Moreover, $X$ can be $G$-equivariantly embedded into a regular
$G$-representation, so $X$ is a $G$-semistable space. More
generally, if $X$ is an arbitrary complex variety, such that the
good quotient (see \cite{BB}) exists, then $X$ is a locally
$G$-semistable space.

Similarly, a complex space with a holomorphic action of a complex
reductive group $G$ such that the analytic Hilbert quotient
exists, is a locally $G$-semistable space. Here this follows from
\cite{Snow}.
\end{example}

\begin{lemma}\label{lemma:HGradientSubspace}
Let $X$ be a locally $G$-semistable space and let $H$ be a closed
compatible subgroup of $G$. Then $X$ is a locally $H$-semistable
space.
\end{lemma}

\begin{proof}
If $X$ is a $G$-semistable space, then it follows from the
definition that it is also an $H$-semistable space. In particular,
the quotient with respect to the action of $H$ exists. Since a
locally $G$-semistable space $X$ is covered by $G$-open subsets
which are $G$-semistable spaces, it follows that the quotient
$X\to\quot XH$ exists. Since a $G$-saturated subset is also
$H$-saturated, the claim follows.
\end{proof}

The notion of the zero fiber of a gradient map does not make sense
for a locally $G$-semistable space. As a substitute, we define
$X_{cc}\subset X$ to be the set of points $x\in X$ such that
$G\cdot x$ is closed and such that $G_x$ is compatible. For a
$G$-semistable space we have $\Mp\subset X_{cc}$. In particular,
in a locally $G$-semistable space each closed $G$-orbit intersects
$X_{cc}$ since it has a neighborhood which is a $G$-semistable
space. Conversely for a locally $G$-semistable space $X$ and an
$x\in X_{cc}$ there exists a $G$-saturated open neighborhood of
$x$ which is a $G$-semistable space such that $x\in \Mp$. For
this, we first observe that there exists a slice model $(G_x,S,V)$
at $x\in X_{cc}$ since by definition there exists a $G$-open
neighborhood of $x$ which admits the structure of a $G$-semistable
space containing $G\cdot x$ as a closed orbit. In particular the
$G$-action on this neighborhood is given by restriction of a
holomorphic $U^\ce$-representation, where $U^\ce$ is complex
reductive and contains $G$ as a compatible subgroup. The complex
analytic Zariski closure $\overline G_x^Z$ of $G_x$ in $U^\ce$ is
compatible, contains $G_x$ and is contained in $(U^\ce)_x$.
Therefore $G\cap\overline G_x^Z=G_x$. Since $V$ is a holomorphic
$\overline G_x^Z$-representation space, we get an embedding
$G\times^{G_x}S\hookrightarrow U^\ce\times^{\overline G_x^Z}V$. By
Lemma~1.16 in \cite{Sjamaar}, there exists a proper holomorphic
$U^\ce$-equivariant embedding of $U^\ce\times^{\overline G_x^Z}V$
into a holomorphic $U^\ce$-representation space such that $[e,0]$
has minimal distance to $0$. Then it follows from the definition
of the gradient map that $[e,0]\in\Mp$. Moreover, $G\times^{G_x}S$
is $G$-locally closed in the holomorphic $U^\ce$-representation.
This shows that the slice neighborhood $GS$ is a $G$-open
neighborhood of $x$ which admits the structure of a $G$-semistable
space such that $x\in\Mp$.

Moreover, we observe that a topological $G$-space $X$ with
topological Hilbert quotient $\pi\colon X\to\quot XG$ is a locally
$G$-semistable space if and only if for each $x\in X$ there exists
a complex reductive group $U^\ce$ containing $G$ as a compatible
subgroup and a $G$-open neighborhood of $x$ in $X$ which is
$G$-equivariantly homeomorphic to $G\times^HS$ where $H$ is a
compatible subgroup of $G$ such that $G\cap\overline H^Z=H$ and
such that $S$ is $H$-locally closed in a holomorphic $\overline
H^Z$-representation space $V$. Here the Zariski closure of $H$ is
taken in $U^\ce$.

\begin{example}
Let $X$ be a smooth $G$-manifold and assume that the action of $G$
is proper. Then the topological Hilbert quotient coincides with
the geometric quotient and at each $x\in X$ there exists a slice
$G\times^{G_x}S\to GS$ where $S$ is an open neighborhood of $0$ in
a $G_x$-representation space $V$ by \cite{Palais}. The isotropy
$G_x$ is compact, so the $G_x$-representation $V$ extends to a
$(G_x)^\ce$-representation $V^\ce$. Then $X$ is a locally
$G$-semistable space since $(G_x)^\ce=\overline G_x^Z$ where the
Zariski closure is taken in $(U_G)^\ce$.
\end{example}

In a $G$-semistable space, a closed $G$-orbit intersects $\Mp$ in
a unique $K$-orbit. In order to describe the intersection $G\cdot
x\cap X_{cc}$, we need the following two lemmas.

\begin{lemma}\label{lemma:normalisator}
Let $H$ be a compatible subgroup of $G$. Then the normalizer
$\nor_G(H)$ of $H$ in $G$ is compatible.
\end{lemma}

\begin{proof}
The normalizer $\nor_G(H)$ is invariant under the Cartan
involution $\theta$, so it suffices to show that is consists only
of finitely many connected components. By \cite{Po}, the group
$H\cdot \zet_G(H)$, where $\zet_G(H)$ denotes the centralizer of
$H$ in $G$, is of finite index in $\nor_G(H)$. Let $U^\ce$ be a
complex reductive group containing $G$ as a compatible subgroup.
Then $\zet_G(H)=\zet_G(\overline H^Z)=G\cap\zet_{\overline
G^Z}(\overline H^Z)$. The centralizer $\zet_{\overline
G^Z}(\overline H^Z)$ is invariant under the Cartan involution
$\theta$ and it has only finitely many connected components since
it is an algebraic group. Therefore it is compatible. Then
$\zet_G(H)$ is compatible since the intersection of two compatible
subgroups is compatible. Thus $H\cdot\zet_G(H)$ consists of only
finitely many connected components and the claim follows.
\end{proof}

\begin{lemma}\label{lemma:compatibleConjugated}
Let $H$ be a compatible subgroup of $G$. Let $g\in G$. Then
$gHg^{-1}$ is compatible if and only if $g\in K\cdot\nor_G(H)$. In
particular, if $gHg^{-1}$ is compatible, then $gHg^{-1}$ and $H$
are conjugate in $K$, so $gHg^{-1}=kHk^{-1}$ for some $k\in K$.
\end{lemma}

\begin{proof}
Let $H=L\exp(\lieq)$ be the Cartan decomposition of $H$. Then for
$k\in K$, we have $kHk^{-1}=kLk^{-1}\exp(\Ad(k)\lieq)$. Since the
adjoint action of $K$ stabilizes $\liep$, we conclude that
$kHk^{-1}$ is compatible. So for $g\in K\cdot \nor_G(H)$, the
group $gHg^{-1}$ is compatible.

Conversely, if $gHg^{-1}$ is compatible, we may assume
$g=\exp(\xi)$ where $\xi\in\liep$. The groups $H$ and $gHg^{-1}$
are $\theta$-stable since they are compatible. Therefore
$gHg^{-1}=\theta(gHg^{-1})=\theta(g)H\theta(g^{-1})$. Explicitly,
we have $\exp(\xi)H\exp(-\xi)=\exp(-\xi)H\exp(\xi)$ or
equivalently $\exp(2\xi)\in\nor_G(H)$. Since $\nor_G(H)$ is
compatible by Lemma~\ref{lemma:normalisator} and since
$\xi\in\liep$, we conclude that $\xi$ is contained in the Lie
algebra of $\nor_G(H)$ which implies $g\in\nor_G(H)$.
\end{proof}

It follows from Lemma~\ref{lemma:compatibleConjugated} that for
$x\in X_{cc}$ the orbit $G\cdot x$ intersects $X_{cc}$ in
$K\cdot\nor_G(G_x)\cdot x$. In particular, for $y\in G\cdot x\cap
X_{cc}$ the isotropy $G_y$ is conjugate to $G_x$ in $K$.

\section{Isotropy Stratification}

\subsection{The Isotropy Stratification Theorem}

Let $X$ be a locally $G$-semistable space and let $\pi\colon
X\to\quot XG$ be the  topological Hilbert quotient. For a closed
compatible subgroup $H$ of $G$ we define \[\xh:=\{x\in
X_{cc};\,G_x=H\}.\] Note that we have a partition
$X_{cc}=\bigcup_H \xh$, where the union is taken over all
compatible isotropy groups of points on closed orbits. We call the
$G$-saturated set
\[I_H(X):=\pi^{-1}(\pi(\xh))=\{x\in X;\,\overline{G\cdot
x}\cap\xh\neq\emptyset\}\] the \emph{$G$-isotropy stratum of $H$
in $X$}. In other words, $I_H(X)$ consists of those fibers of
$\pi$, where the unique closed orbit is of orbit type $G/H$. We
abbreviate $I_H(X)$ by $I_H$ and call $I_H$ a \emph{stratum}, if
no confusion is possible.

\begin{theorem}[Isotropy Stratification Theorem]\label{theorem:IsotropyStratification}
Let $X$ be a locally $G$-semistable space. Let $\mathcal I$ be an
index set and let $H_i$, $i\in\mathcal I$, be compatible subgroups
of $G$ such that $\{H_i;\, i\in\mathcal I\}$ is
    a set of representatives of conjugacy classes of
    isotropy groups of closed $G$-orbits in
    $X$. Then
\begin{enumerate}
\item $I_{H_i}$ is $G$-saturated and locally
closed.\label{thm:IsoStrat:1} \item
$X=\bigcup_{i\in\mathcal I}
    I_{H_i}$ and the union is
    disjoint and locally
    finite.\label{thm:IsoStrat:3} \item If
    $\overline{I_{H_i}}\cap I_{H_j}\neq \emptyset$ and
    $I_{H_i}\neq I_{H_j}$ then there exists a $g\in G$ with
    $gH_ig^{-1}\lvertneqq H_j$.\label{thm:IsoStrat:5}
\end{enumerate}
\end{theorem}

\begin{example}
Consider the adjoint representation of a connected semisimple
group $G$. There exist finitely many $\theta$-stable Cartan
subalgebras $\lieh_1,\ldots,\lieh_n$ in the Lie algebra $\lieg$ of
$G$ such that each Cartan subalgebra is conjugate to one of these.
A $G$-orbit in $\lieg$ is closed if and only if it intersects a
Cartan subalgebra. If $\xi\in \lieh_i$ is a regular element, then
the isotropy $G_\xi$ equals the centralizer $H_i:=Z_G(\lieh_i)$ of
$\lieh_i$ in $G$. Since a neighborhood of $\xi$ in $\lieh_i$ is a
slice for the $G$-action, we conclude that the strata
$I_{H_1},\ldots,I_{H_n}$ are the open strata in $\lieg$.

In the special case where $G$ is complex reductive, all Cartan
subalgebras are conjugate to each other, hence there exists an
open and dense stratum.
\end{example}

\subsection{The proof of Theorem~\ref{theorem:IsotropyStratification}}

First note that $I_{H_i}$ is $G$-saturated by definition and that
$X$ is the union of the strata since each orbit contains a closed
orbit in its closure and since each closed orbit intersects
$X_{cc}$. The union is disjoint for if two strata $I_H$ and
$I_{H'}$ intersect, the intersection contains a closed orbit,
which implies that $H$ and $H'$ are conjugate in $G$.

Locally the stratification is determined by the stratification of
a slice. For this, let $(G_x,S,V)$ be a slice model at $x\in
X_{cc}$. Since $S$ is a $G_x$-semistable space, we have the notion
of $G_x$-isotropy strata in $S$. Moreover, a $G$-orbit $G\cdot y$
with $y\in S$ is closed in $X$ if and only if the $G_x$-orbit
$G_x\cdot y$ is closed in $S$. Note also that $(G_x)_y=G_y$.
Identifying $GS$ with $G\times^{G_x}S$, we get
\[I_H(X)\cap GS=\bigcup_{H_i}G\times^{G_x}I_{H_i}(S),\] for a compatible subgroup $H$ of $G$. Here the
union is taken over all subgroups $H_i$ of $G_x$ such that $H_i$
is conjugate to $H$ in $G$. In particular, we obtain
$I_{G_x}(X)\cap GS=G\cdot I_{G_x}(S)$. The stratum $I_{G_x}(S)$ is given by
the intersection of $S$ with $I_{G_x}(V)$. But the closed $G_x$-orbits
in $V$ with isotropy conjugate to $G_x$ are the $G_x$-fixed points
$V^{G_x}$ and the orbits which contain a fixed point in their closure
are orbits through points $v\in V$ which are the sum of a fixed
point and an element of the nullcone $\mathcal N:=\{v\in V;
0\in\overline{G_x\cdot v}\}$. This implies $I_{G_x}(S)=S\cap (V^{G_x}+\mathcal
N)$, where $V^{G_x}+\mathcal N:=\{v_1+v_2; v_1\in V^{G_x}, v_2\in\mathcal
N\}$. Since the nullcone is real algebraic in $V$ (Lemma~7.1 in
\cite{Schuetzdeller}), this shows that $I_{G_x}(S)$ is locally
closed which in turn implies that $I_{G_x}(X)$ is locally closed.

We show that the stratification is locally finite. If $X$ is a
$G$-representation space, the strata are cones which by definition
means that they are invariant under multiplication with positive
real numbers. Then the stratification is locally finite at $0\in
X$ if and only if there exist only finitely many strata in $X$. An
arbitrary locally $G$-semistable space is covered by slice
neighborhoods and the strata in a slice neighborhood are
determined by the strata in a slice. A slice is a subspace of a
representation space, so local finiteness follows from

\begin{proposition}\label{proposition:RepresentationFinite}
Assume $X$ is a $G$-representation which is given by restriction
of a holomorphic $U^\ce$-representation. Then there exist only
finitely many strata in $X$.
\end{proposition}

\begin{proof}
Since the representation is completely reducible
(Lemma~\ref{lemma:vollstReduzibel}), we have an invariant
decomposition $X=W\oplus X^G$. Replacing $X$ by $W$, we may assume
$X^G=\{0\}$.

Recall that we have the notion of a gradient map and its zero
fiber $\Mp$ on the $G$-semistable space $X$. The strata in $X$ are
determined by the orbit types of closed orbits in $X$ and the
closed orbits intersect $\Mp$. But the strata as well as $\Mp$ are
cones in $X$, so up to the nullcone $I_G(X)$ every stratum
intersects $S^n\cap\Mp$ where $S^n$ is a sphere which is defined
with respect to an arbitrary inner product. Therefore it suffices
to show that $\Mp\cap S^n$ intersects only finitely many strata.
Let $(G_v,S,V)$ be a slice model at $v\in\Mp\cap S^n$. Since the
compact set $\Mp\cap S^n$ is covered by finitely many slice
neighborhoods, it suffices to show that the slice neighborhood
$GS$ consists of only finitely many strata. But the $G$-isotropy
strata in $GS$ are determined by the $G_v$-isotropy strata in $S$
which in turn are determined by the $G_v$-isotropy strata in the
representation space $V$. Moreover $G_v$ is a proper compatible
subgroup of $G$ since $X^G=\{0\}$. Therefore the proposition
follows by induction over the dimension and the number of
connected components of $G$.
\end{proof}

\begin{remark}
If $G$ acts properly on $V$, then there exists an open and dense
stratum in $V$. The proof is similar to that of
Proposition~\ref{proposition:RepresentationFinite}. Since every
$G$-orbit is closed, the nullcone is trivial and the sphere $S^n$
can be covered by finitely many slice neighborhoods where each
slice neighborhood contains a dense stratum by induction. If the
dimension of $V$ is greater than one, then the claim follows since
$S^n$ is connected. If the dimension of $V$ equals one, then there
is a dense stratum since the strata are cones.

More generally, if $X$ is a smooth $G$-manifold such that the
$G$-action is proper and if the quotient $X/G$ is connected, then
there exists a dense stratum in $X$.
\end{remark}

It remains to show the last statement of
Theorem~\ref{theorem:IsotropyStratification}. For this, let
$x\in\overline{I_{H_i}}\cap I_{H_j}$. Since the intersection is
$G$-saturated, we may assume that $G\cdot x$ is closed and that
$G_x=H_j$. Let $(H_j,S,V)$ be a slice model at $x$. The slice
neighborhood $GS$ and $\overline{I_{H_i}}$ are $G$-saturated, so
$GS$ intersects $X^{<H_i>}$. But this implies that $H_i$ is
conjugate to a subgroup of $H_j$. Thus the proof of
Theorem~\ref{theorem:IsotropyStratification} is completed.

\subsection{Strata and Slices}\label{subsection:strataSlices}

Let $x\in X_{cc}$ and let $(G_x,S,V)$ be a slice model at $x$. We
will give a criterion for which choice of $H_i$ the stratum
$I_{H_i}(S)$ in $S$ is non-empty in
Lemma~\ref{lemma:nonEmptyStrata}. For this, we first observe that
locally the defining set $\xh$ of $I_H(X)$ is determined by
$S^{<H>}=\{y\in S; (G_x)_y=H, G_x\cdot y\text{ closed}\}$.

\begin{lemma}\label{lemma:xhImSlice}
Let $(G_x,S,V)$ be a slice model at $x\in X_{cc}$ and let $H$ be a
compatible subgroup of $G_x$. Then there exists an open
$\nor_G(H)$-invariant neighborhood $W$ of $x$ in $X$ which
contains $S$ such that\begin{enumerate} \item $W\cap
X^{<H>}=\nor_G(H)\cdot S^{<H>}$ and \item $W\cap
X^{H}=\nor_G(H)\cdot S^{H}$.\end{enumerate}
\end{lemma}

\begin{proof}
Let $p\colon G\times^{G_x}S \to G/G_x$ denote the projection
$p([g,s])=g\cdot G_x$. The set of $H$-fixed points in the
homogeneous space $G/G_x$ consists of isolated $\nor_G(H)$-orbits.
Therefore there exists an open $\nor_G(H)$-invariant neighborhood
$W'$ of $e\cdot G_x$ in $G/G_x$ such that the $H$-fixed points in
$W'$ are given by $\nor_G(H)\cdot e\cdot G_x$. Defining
$W:=p^{-1}(W')$, the claim follows from $G$-equivariance of $p$.
\end{proof}

In the rest of this paper we assume that \emph{every non-empty
$G_x$-isotropy stratum in the slice $S$ contains $x$ in its
closure}. This can always be achieved by replacing $S$ by an
appropriate open neighborhood of $x$. More precisely, we remove
the closures of the strata which do not contain $x$ in their
closure. By Proposition~\ref{proposition:RepresentationFinite}
this is only a finite number of strata. Then we choose a
$G_x$-open neighborhood of $x$ inside the so obtained
$G_x$-invariant open neighborhood. This is possible by
Corollary~\ref{corollary:SaturatedNeighborhood}.

\begin{lemma}\label{lemma:nonEmptyStrata}
Let $X$ be a locally $G$-semistable space, let $x\in X_{cc}$ and
let $(G_x,S,V)$ be a slice model at $x$. Then a $G_x$-isotropy
stratum $I_{H}(S)$ in $S$ is non-empty if and only if $x$ is
contained in the closure $\xha$ of $\xh$.
\end{lemma}

\begin{proof}
First, if $x\in\xha$ then $S^{<H>}$ is non-empty by
Lemma~\ref{lemma:xhImSlice} which implies that $I_H(S)$ is
non-empty.

Recall that $S$ is a $G_x$-semistable space with gradient map
$\mu_{\liep_x}\colon S\to \liep_x$. Since $x$ is a $G_x$-fixed
point, the orbit $G_x\cdot x=x$ is closed and a fortiori $x$ is
contained in the zero-fiber $\mathcal M_{\liep_x}$ of
$\mu_{\liep_x}$. If $I_H(S)$ is a non-empty stratum, we have
$x\in\overline{I_H(S)}$. Then $x\in\overline{I_H(S)\cap\mathcal
M_{\liep_x}}$ by Corollary~\ref{corollary:MpSchnitt}. Thus there
exists a sequence $x_n\in I_H(S)\cap \mathcal M_{\liep_x}$ which
converges to $x$. The $G_x$-isotropy at $x_n$ is compatible and
conjugate to $H$. Therefore there exist $k_n\in K_x$ with
$k_nx_n\in S^{<H>}\subset\xh$ by
Lemma~\ref{lemma:compatibleConjugated}. But $k_nx_n$ converges to
the $K_x$-fixed point $x$ so we get $x\in\xha$.
\end{proof}

\subsection{The splitting number}

In this section we assume that $X$ contains a dense stratum $I_H$.
Let $x\in X_{cc}$ and let $(G_x,S,V)$ be a slice model at $x$.
Recall that we assume that every open $G_x$-isotropy stratum
contains $x$ in its closure. We define $n(x)$ to be the number of
open $G_x$-isotropy strata in $S$. For an arbitrary point $y\in X$
we define $n(y):=n(x)$ where $x$ is a point in the fiber
$\pi^{-1}(\pi(y))$ of the quotient $\pi\colon X\to\quot XG$ which
is contained in $X_{cc}$. We call $n(y)$ the \emph{splitting
number at $y$}.

The splitting number is well defined. In order to see that, let
$x,x'\in\pi^{-1}(\pi(y))\cap X_{cc}$. Then $x'=gx$ for a $g\in G$.
If $S'$ is a slice at $x'$ then $S:=g^{-1}S'$ is a slice at $x$
and the map $g^{-1}\colon S'\to S$, $s\mapsto g^{-1}s$ sets up a
one to one correspondence between open strata in $S'$ and open
strata in $S$. The following proposition implies that $n(x)$ does
not depend on the choice of the slice at $x$.

\begin{proposition}\label{proposition:openStrata}
Let $X$ be a locally $G$-semistable space which contains a dense
stratum $I_H(X)$. Let $(G_x,S,V)$ be a slice model at $x\in
X_{cc}$. Then $I_{H'}(S)$ is a non-empty open stratum in $S$ if
and only if $kH'k^{-1}=H$ and $kx\in \xha$ for a $k\in K$. The
union of the open strata is dense in $S$ and coincides with
$I_H(X)\cap S$.
\end{proposition}

\begin{proof}
By Lemma~\ref{lemma:nonEmptyStrata} the stratum $I_{k^{-1}Hk}(S)$
is non-empty if and only if $x\in\overline{X^{<k^{-1}Hk>}}$, or
equivalently $kx\in\xha$.

Let $I_{H'}(S)$ be any stratum in $S$ which intersects $I_H(X)\cap
S$. Then $H'$ is conjugate to $H$ in $G$ and we get
$I_{H'}(S)\subset I_H(X)\cap S$. It follows from
Theorem~\ref{theorem:IsotropyStratification} (3) that $I_{H'}(S)$
is closed in $I_H(X)\cap S$. The $G_x$-isotropy stratification of
$V$ is finite by
Proposition~\ref{proposition:RepresentationFinite}, which implies
that every stratum in $I_H(X)\cap S$ is also open in $S$. In
particular the union of the open strata coincides with the open
and dense subset $I_H(X)\cap S$.

Finally, a stratum $I_{H'}(S)$ is open if and only if it is
contained in $I_H(X)\cap S$ which is the case if and only if $H'$
is conjugate to $H$ in $G$ and then also in $K$ by
Lemma~\ref{lemma:compatibleConjugated}.
\end{proof}

For $x\in\xh$, the group $kHk^{-1}$ is a subgroup of $G_x=H$ if and only if
$k\in\nor_K(H)$. This implies

\begin{corollary}\label{corollary:n=1aufIH}
Let $X$ be a locally $G$-semistable space which contains a dense
stratum $I_H(X)$. Then $n(x)=1$ for all $x\in I_H$.
\end{corollary}

\begin{example}\label{example:Splitting1}
By the remark following
Proposition~\ref{proposition:RepresentationFinite}, the splitting
number is constantly one if $X$ is a smooth $G$-space such that
$G$ acts properly on $X$.

If $X$ is a complex space or a complex variety equipped with a
regular action of a complex reductive group $G$, then it follows
from Luna's Slice Theorem (\cite{Luna}) that a stratum is locally
closed with respect to the Zariski topology. Then, if $X$ is
irreducible, there exists a dense stratum in $X$. In particular, if $X$ is
normal, it is irreducible at each point which implies $n(x)=1$ for
all $x\in X$.
\end{example}

\section{The Restriction Theorem}

\subsection{The Restriction Theorem}

Let $X$ be a locally $G$-semistable space containing a dense
stratum $I_H(X)$. Then $X$ is a locally $\nor_G(H)$-semistable
space by Lemma~\ref{lemma:HGradientSubspace} which in turn implies
that the closed $\nor_G(H)$-invariant subset $\xha$ of $X$ is a
locally $\nor_G(H)$-semistable space. In particular the quotient
$\pi_N\colon\xha\to\quot{\xha}{\nor_G(H)}$ exists.

We now state and prove our main result.

\begin{Rtheorem}\label{theorem:RestrictionThm}
Let $X$ be a locally $G$-semistable space containing a dense
$G$-isotropy stratum $I_H$. Then the inclusion
$\xha\hookrightarrow X$ induces a continuous finite surjective map
\[\Phi\colon\quot{\xha}{\nor_G(H)}\to \quot{X}{G}.\] For $x\in X$,
the number of points in the fiber $\Phi^{-1}(\pi(x))$ is equal to
the splitting number $n(x)$. If $x\in\xha$ with $n(x)>1$, then
$\Phi$ is not open at $\pi_N(x)$.
\end{Rtheorem}

Here, by a finite map, we mean a proper map with finite fibers.

\begin{remark}
If $X$ is a smooth manifold which is covered by $G$-open sets
which are diffeomorphic to $G$-locally closed submanifolds of
holomorphic representation spaces, then $\xha$ is a closed
submanifold of $X$. For the proof, we refer the reader to
Section~\ref{section:xhaGlatt}.
\end{remark}

Since a finite map is closed and a closed bijective map is open,
we get

\begin{corollary}\label{corollary:PhiHomeo}
Assume that $n(x)=1$ for all $x\in X$. Then $\Phi$ is a homeomorphism.
\end{corollary}

For a locally $G$-semistable space $X$, a stratum $I_H$ and the
closure $\overline{I_H}$ are again locally $G$-semistable spaces.
In particular, if $X$ does not contain a dense stratum, we may
apply the Restriction Theorem for $\overline{I_H}$. Moreover, for
each stratum $I_H$, the restriction $\Phi\colon
\quot{\xh}{\nor_G(H)}\to\quot {I_H}G$ is a homeomorphism by
Corollary~\ref{corollary:n=1aufIH}.

\begin{example}
$\Phi$ is a homeomorphism in the situations
considered in Example~\ref{example:Splitting1}.

If $X$ is a smooth $G$-manifold such that the quotient $X/G$ is
connected and if the $G$-action is proper, then a dense stratum
$I_H$ always exists. In particular, if $G$ is compact, the
assertion of the Restriction Theorem is equivalent to a result in
\cite{Schwarz80}.

For the action of a complex reductive group $G$ on an irreducible
normal complex affine variety $X$, a result similar to our
Restriction Theorem was established in \cite{LunaRich}. Here the
map $\quot{X^H}{\nor_G(H)}\to\quot XG$ is considered and it is
assumed that $\quot {X^H}{\nor_G(H)}$ is irreducible. However,
$\xha$ is a union of irreducible components of $X^H$ so the
quotients $\quot {X^H}{\nor_G(H)}$ and $\quot {\xha}{\nor_G(H)}$
coincide under the additional assumption that $\quot
{X^H}{\nor_G(H)}$ is irreducible. Due to normality of $X$, the map
$\Phi$ is an isomorphism of affine varieties.
\end{example}

\begin{example}
Let $n\geq 2$, $G:=\Sl_{2n}(\er)$ and consider the natural action
of $J:=\So(n,n)$ on $V:=\er^{2n}$. Defining $X:=G\times^JV$, a
$G$-orbit $G\cdot [g,v]$ in $X$ is closed if and only if the
$J$-orbit $J\cdot v$ is closed in $V$. Computations show that each
closed $J$-orbit in $V$ intersects one of the rays
$\ell_1:=\er^{\geq 0}\cdot (e_1,0)$ and $\ell_2:=\er^{\geq 0}\cdot
(0,e_1)$. In this notation we identify $\er^{2n}$ and
$\er^n\times\er^n$. Therefore the quotient $\quot XG$ is
homeomorphic to $\ell_1\cup\ell_2$ which is homeomorphic to $\er$.

Let $H_1:=J_{(e_1,0)}$ and $H_2:=J_{(0,e_1)}$. The
$J$-representation $V$ consists of three strata, namely the
nullcone $I_J(V)$ and the open strata $I_{H_1}(V)=\{(v_1,v_2)\in
V;\|v_1\|>\|v_2\|\}$ and $I_{H_2}(V)=\{(v_1,v_2)\in
V;\|v_1\|<\|v_2\|\}$. In particular, the splitting number at
$[e,0]$ equals $2$. The groups $H_1$ and $H_2$ are conjugate in
$K$. Explicitly we have
$H_1=k_0H_2k_0^{-1}$ for $k_0=\left(%
\begin{array}{cc}
  0 & -I_n \\
  I_n & 0 \\
\end{array}%
\right)$. Defining $H:=H_1$, we conclude that the stratum
$I_{H}(X)=I_{H_2}(X)$ is dense in $X$ and we may apply the
Restriction Theorem. We get
\[\xha=\{[g,v];\,g\in\nor_G(H),
v\in\ell_1\}\cup\{[gk_0,v];\,g\in\nor_G(H),v\in\ell_2\}.\] Then
the quotient $\quot{\xha}{\nor_G(H)}$ is homeomorphic to
$\{[e,v];\,v\in\ell_1\}\cup\{[k_0,v];\,v\in\ell_2\}$, which is
homeomorphic to the disjoint union of two rays. With respect to
these identifications, the map $\Phi$ is given by $\Phi([e,v])=v$
for $v\in\ell_1$ and $\Phi([k_0,v])=v$ for $v\in\ell_2$. We
observe that $\Phi$ glues the two rays at their boundary points.
\end{example}

\subsection{The proof of the Restriction Theorem}

First note that we may assume that $X$ is a $G$-semistable space.
Then the quotient $\quot XG$ is homeomorphic to $\Mp/K$ where
$\Mp$ is the zero fiber of the gradient map $\mu_\liep\colon
X\to\liep$. Moreover the quotient $\quot{\xha}{\nor_G(H)}$ is
homeomorphic to $\Mnp/\nor_K(H)$. Here $\np$ denotes the
intersection of $\liep$ with the Lie algebra of $\nor_G(H)$ and
$\Mnp$ is the zero fiber of the $\nor_G(H)$-gradient map
$\mu_{\np}\colon\xha\to\np$.

The following lemma shows in particular that $\Mnp$ is contained
in $\Mp$ which implies that the map $\Phi$ corresponds to a map
$\phi\colon\Mnp/\nor_K(H)\to\Mp/K$ such that the diagram
\[ \xymatrix{
\Mnp/\nor_K(H)\ar^\sim[d] \ar[r]^{\hspace{0.5cm}\phi}    & \Mp/K\ar^\sim[d] \\
\quot{\xha}{\nor_G(H)}\ar[r]^{\hspace{0.5cm}\Phi} & \quot{X}{G}}\]
commutes.

\begin{lemma}\label{lemma:Mp=Mnp}
We have $\Mnp=\xha\cap\Mp$.
\end{lemma}

\begin{proof}
Since $\np$ is a subspace of $\liep$, the inclusion
$\xha\cap\Mp\subset\Mnp$ follows from the definition of the
gradient map.

We show $\xh\cap\Mnp=\xh\cap\Mp$. For this let $x\in\xh\cap\Mnp$.
Since $G\cdot x$ is closed, we have $gx\in\Mp$ for some $g\in G$.
The isotropy groups $G_x=H$ and $G_{gx}=gHg^{-1}$ are compatible,
which yields $g=kh\in K\cdot \nor_G(H)$ by
Lemma~\ref{lemma:compatibleConjugated}. The $K$-invariance of
$\Mp$ implies $hx\in\Mp\subset\Mnp$. But then $h\cdot x
\in\nor_K(H)\cdot x$ since $\nor_G(H)\cdot x$ intersects $\Mnp$ in
a unique $\nor_K(H)$-orbit. Thus $x\in\Mp$ follows from the
$K$-invariance of $\Mp$. With Corollary~\ref{corollary:MpSchnitt}
we conclude $\Mnp=\overline{\xh\cap\Mnp}\subset\xha\cap\Mp$.
\end{proof}

Surjectivity of $\phi$ and then also of $\Phi$ follows from

\begin{lemma}\label{lemma:SurjektivitaetDerAbbZwischenDenQuotienten}
We have $\Mp=K\cdot \Mnp$. In particular $\Mp\subset K\cdot\xha$.
\end{lemma}

\begin{proof}
The inclusion $K\cdot \Mnp\subset \Mp$ follows from $K$-invariance
of $\Mp$ and Lemma~\ref{lemma:Mp=Mnp}.

For $x\in\Mp\cap I_H$, the isotropy group $G_x$ is compatible and
conjugate to $H$. Then it is conjugate to $H$ in $K$ by
Lemma~\ref{lemma:compatibleConjugated}. This shows $\Mp\cap
I_H=K\cdot (\Mp\cap \xh)$ and $\Mp\cap I_H\subset K\cdot\Mnp$
follows from Lemma~\ref{lemma:Mp=Mnp}.

We claim that $\Mp\cap I_H$ is dense in $\Mp$. Then the assertion
of the lemma follows since $K\cdot\Mnp$ is closed. If this is not
the case, then there exist an $x\in\Mp$ and an open neighborhood
$W$ of $x$ in $\Mp$ which does not intersect $I_H$. Recall that
the restriction $\pi\colon\Mp\to\quot XG$ of the topological
Hilbert quotient is open. Therefore $\pi^{-1}(\pi(W))$ is a $G$-open
neighborhood of $x$ in $X$ which does not intersect $I_H$. This is
a contradiction, since $I_H$ is dense in $X$. So $\Mp\cap I_H$ is
dense in $\Mp$.
\end{proof}

For $x\in\Mnp$ the number of points in the fiber
$\phi^{-1}(\pi(x))$ equals the number of $\nor_K(H)$-orbits in
$K\cdot x\cap\Mnp$. Here we describe the 1-1-correspondence
between these orbits and the open $G_x$-isotropy strata in a slice
at $x$ explicitly. As a consequence we see that the number of
points in the fiber $\phi^{-1}(\pi(x))$ is equal to the splitting
number $n(x)$.

\begin{proposition}\label{proposition:fiberSplitting}
Let $x\in\Mnp$ and let $(G_x,S,V)$ be a slice model at $x$. Then
\begin{align*}
\Psi\colon(K\cdot x\cap\Mnp)/\nor_K(H)&\to
\{\mbox{Non-empty open $G_x$-isotropy strata in }S\},\\
\nor_K(H)\cdot k\cdot x&\mapsto I_{k^{-1}Hk}(S)
\end{align*}
is well-defined and bijective.
\end{proposition}

\begin{proof}
First, we show that $\Psi$ is well-defined. For
$kx\in\Mnp\subset\xha$, the stratum $I_{k^{-1}Hk}(S)$ is non-empty
and open by Proposition~\ref{proposition:openStrata}. Assume that
$\nor_K(H)\cdot k_1 x=\nor_K(H)\cdot k_2 x\subset K\cdot x\cap
\xha$ with $k_1,k_2\in K$. This is equivalent to
$k_1\in\nor_K(H)\cdot k_2\cdot K_x$ which in turn is equivalent to
the condition that $k_1^{-1}Hk_1$ and $k_2^{-1}Hk_2$ are conjugate
in $K_x$. But then $k_1^{-1}Hk_1$ and $k_2^{-1}Hk_2$ define the
same $G_x$-isotropy stratum in $S$.

For injectivity, assume $\Psi(\nor_K(H)\cdot k_1
x)=\Psi(\nor_K(H)\cdot k_2 x)$. Then
$I_{k_1^{-1}Hk_1}(S)=I_{k_2^{-1}Hk_2}(S)$ and the compatible
groups $k_1^{-1}Hk_1$ and $k_2^{-1}Hk_2$ are conjugate in $G_x$.
By Lemma~\ref{lemma:compatibleConjugated} they are conjugate in
$K_x$. Thus we get $\nor_K(H)\cdot k_1 x=\nor_K(H)\cdot k_2 x$.

It remains to show that $\Psi$ is surjective. By
Proposition~\ref{proposition:openStrata} a non-empty open stratum
is of the form $I_{k^{-1}Hk}(S)$ for some $k\in K$ with
$kx\in\xha$. Then $kx\in\Mnp$ by Lemma~\ref{lemma:Mp=Mnp}
and surjectivity is proved.
\end{proof}

The inclusion $\Mnp\hookrightarrow \Mp$ is continuous and
proper. Since $\nor_K(H)$ and $K$ are compact, this implies that
$\phi$ is continuous and proper. Hence, $\phi$ is finite.

To prove the last assertion of the Restriction Theorem let
$x,y\in\Mnp/\nor_K(H)$ with $x\neq y$ and $\phi(x)=\phi(y)$. Let
$W_x$ and $W_y$ be open neighborhoods of $x$ and $y$,
respectively, such that $W_x\cap W_y=\emptyset$. Assume that
$\phi$ is open at $x$. Since $\Mnp\cap\xh$ is dense in $\Mnp$ by
Corollary~\ref{corollary:MpSchnitt}, there exists a $z\in W_y\cap
(\xh\cap\Mnp)/\nor_K(H)$ satisfying $\phi(z)\in\phi(W_x)$. But
then $\phi(z)\in\phi(W_x)\cap\phi(W_y)$, which is impossible since
the restriction $\phi\colon(\xh\cap\Mnp)/\nor_K(H)\to (I_H\cap
\Mp)/K$ is injective by Corollary~\ref{corollary:PhiHomeo}.

\section{Smoothness of $\xha$}\label{section:xhaGlatt}

We assume that $X$ is a smooth locally $G$-semistable space and
that $I_H$ is a dense stratum in $X$. The purpose of this section
is to show that then $\xha$ is smooth.

\begin{theorem}
Assume $X$ is a smooth locally $G$-semistable space containing a
dense stratum $I_H(X)$. Then the closure $\xha$ of $\xh$ is open
and closed in the fixed point set $X^H$. In particular, $\xha$ is
a closed submanifold of $X$.
\end{theorem}

\begin{proof}
First, we reduce the assertion of the theorem to the case, where
$X$ is a $G$-representation space. Assuming that $X$ is a
$G$-semistable space, the quotient $\quot{\xha}{\nor_G(H)}$ is
homeomorphic to $\Mnp/\nor_K(H)$. Moreover, we have
$\Mnp\subset\Mp$ by Lemma~\ref{lemma:Mp=Mnp}. Since $\xha$ and
$X^H$ are $\nor_G(H)$-invariant, it therefore suffices to show
that $\xha$ is open in $X^H$ at a point $x\in\xha\cap\Mp$. Let
$(G_x,S,V)$ be a slice model at $x$. Locally near $x$, we have
$\xh=\nor_G(H)\cdot S^{<H>}$ and $X^H=\nor_G(H)\cdot S^H$ by
Lemma~\ref{lemma:xhImSlice}. Then it suffices to show that
$\overline{V^{<H>}}=V^H$ since $S$ is an open neighborhood of $0$
in $V$.

By \cite{St09} there exists a subset $\mathcal U$ of $V^H$ which
is open with respect to the real algebraic Zariski topology such
that $G_x\cdot v$ is closed for $v\in \mathcal U$. Furthermore,
the set $\mathcal O:=\{v\in V^H; \dim G_x\cdot v\geq \dim G_x/H\}$
is Zariski open in $V^H$. The intersection $\mathcal U\cap
\mathcal O$ contains $V^{<H>}=I_H(V)\cap V^H$. If $V^{<H>}$ is not
dense in $V^H$, there exists a stratum $I_{H'}(V)$ such that the
intersection $I_{H'}(V)\cap\mathcal U\cap\mathcal O$ contains an
interior point $v_0$ in $V^H$. Conjugating $H'$ if necessary, we
may assume that $H'$ contains $H$ as an open subgroup. By
Proposition~\ref{proposition:openStrata}, it now suffices to show
that $I_{H'}(V)$ is open in $V$.

Let $(H',S_0,W_0)$ be a slice model at $v_0$ and let $(H,S_1,W_1)$
be a slice model at $v_1\in V^{<H>}$. Then $W_0$ and $W_1$ are
equivalent as $H$-representation spaces since $V=T_{v_0}(G\cdot
v_0)\oplus W_0=T_{v_1}(G\cdot v_1)\oplus W_1$ are $H$-invariant
decompositions of $V$ and since since $T_{v_0}(G\cdot v_0)$ and
$T_{v_1}(G\cdot v_1)$ are equivalent $H$-representations. Define
$W:=W_0\cong W_1$. We have $W=W^H+\mathcal N$ where $\mathcal N$
is the $H$-nullcone in $W$ since $I_H(V)$ is open. For openness of
$I_{H'}(V)$ we must show $W=W^{H'}+\mathcal N'$ where $\mathcal
N'$ is the $H'$-nullcone. But we have $\mathcal N=\mathcal N'$
since $H$ is open in $H'$. Since $v_0$ is an interior point of
$I_{H'}(V)\cap V^H$ in $V^H$, there exists a neighborhood $D$ of
$0$ in $W$ such that $D\cap W^H\subset W^{H'}+\mathcal N$. By
algebraicity we get $W^H\subset W^{H'}+\mathcal N$. But this
implies $W=W^H+\mathcal N=W^{H'}+\mathcal N'$ and the proof is
completed.
\end{proof}

%%%%%%%%%%%%%%%%%%%%%%%%%%%%%%%%%%

\newcommand{\noopsort}[1]{} \newcommand{\printfirst}[2]{#1}
\newcommand{\singleletter}[1]{#1} \newcommand{\switchargs}[2]{#2#1}
\providecommand{\bysame}{\leavevmode\hbox to3em{\hrulefill}\thinspace}


\begin{thebibliography}{H}

%\bibitem[ALQ95]{ALQ} H.~Azad, J.~J.~Loeb and M.~N.~Qureshi, \emph{Totally real orbits in affine quotients of reductive groups}, Nagoya Math. J. \textbf{139} (1995), 87--92.

%\bibitem[Bir71]{Birkes} D.~Birkes, \emph{Orbits of linear algebraic groups}, Ann.
%of Math. (2) \textbf{93} (1971), 459--475.

%\bibitem[Bre72]{Bredon} G.~E.~Bredon, \emph{Introduction to
%    compact transformation groups}, volume~46 of \emph{Pure and Applied Mathematics},
%    Academic Press Inc., New York, 1972.

%\bibitem[BS00]{Sch} M.~Brion and G.~W.~Schwarz, \emph{Th\'eorie des invariants \& G\'eom\'etrie des vari\'et\'es %quotients}, Travaux en cours, vol.~61, Hermann, Paris, 2000.

\bibitem[BCM02]{BB} A.~Bia{\l}ynicki-Birula, J.~B.~Carrell and W.~M.~McGovern, \emph{Algebraic quotients. Torus actions and cohomology. The adjoint representation and the adjoint action}, Encyclopaedia of Mathematical Sciences 131, Invariant Theory and Algebraic Transformation Groups, II. Springer-Verlag, Berlin, 2002.

%\bibitem[GH08]{GrebH} D.~Greb and P.~Heinzner, \emph{K\"ahlerian reduction in steps}, Progress in Mathematics, to %appear.

%\bibitem[GR84]{GR} H.~Grauert and R.~Remmert, \emph{Coherent Analytic Sheaves}, Grundlehren der Mathematischen %Wissenschaften [Fundamental Principles of Mathematical Sciences] 265, Springer-Verlag, Berlin, 1984.

%\bibitem[HC58]{HarishChandra} Harish-Chandra, \emph{Spherical functions on a semisimple Lie group, I.}, Amer. J. %Math. \textbf{80} (1958), 241--310.

%\bibitem[He91]{H} P.~Heinzner, \emph{Geometric invariant theory on Stein spaces}, Math. Ann. \textbf{289} (1991), no.~4, 631--662.

%\bibitem[HM01]{HMig} P.~Heinzner and L.~Migliorini, \emph{Projectivity of moment map quotients}, Osaka J. Math. %\textbf{38} (2001), no.~1, 167--184.

%\bibitem[HMP98]{HMP} P.~Heinzner, L.~Migliorini and M.~Polito, \emph{Semistable quotients}, Ann. Scuola Norm. Sup. %Pisa Cl. Sci. (4) \textbf{26} (1998), no.~2, 233--248.

\bibitem[HS07a]{Schuetzdeller} P.~Heinzner and P.~Sch\"{u}tzdeller, \emph{Convexity properties of gradient maps}, arXiv:0710.1152v1 [math.CV], 2007.

\bibitem[HS07b]{HSch} P.~Heinzner and G.~Schwarz, \emph{Cartan decomposition of the moment map}, Math. Ann. \textbf{337} (2007), 197--232.

%\bibitem[HS07c]{HSt} P.~Heinzner and H.~St\"otzel,
%    \emph{Semistable
%    points with respect to real forms}, Math. Ann. \textbf{338}
%    (2007), 1--9.

%\bibitem[HSS08]{HSS} P.~Heinzner, G.~Schwarz and H.~St\"otzel,
%\emph{Stratifications with respect to actions of real reductive
%groups}, Compositio Math. \textbf{144} (2008), 163--185.

%\bibitem[Hel62]{Hel} S.~Helgason, \emph{Differential
%    geometry and symmetric spaces}, Pure and Applied
%    Mathematics, Vol. XII, Academic Press, New York, 1962.

%\bibitem[Hel78]{Helgason} S.~Helgason, \emph{Differential
%    geometry, Lie groups, and symmetric spaces}, Pure and Applied
%    Mathematics 80, Academic Press Inc., New York, 1978.

%\bibitem[HN91]{HN}
%J.~Hilgert and K.-H.~Neeb, \emph{Lie-Gruppen und Lie-Algebren}, Vieweg, 1991.

\bibitem[Ho65]{Hochschild} G.~Hochschild, \emph{The structure of Lie groups}, Holden-Day Inc., San Francisco, 1965.

%\bibitem[J\"a68]{Jaenich} K.~J\"anich,
%    \emph{Differenzierbare
%    $G$-Mannigfaltigkeiten}, Lecture Notes in Mathematics, No
%    59. Springer-Verlag, Berlin, 1968.

%\bibitem[Kn02]{Knapp} A.~W.~Knapp, \emph{Lie groups beyond an introduction}, Second edition, Progress in Mathematics 140, Birkh\"auser Boston Inc., Boston, MA, 2002.

%\bibitem[Kr84]{Kraft} H.~Kraft, \emph{Geometrische Methoden
%    in der Invariantentheorie}, Aspects of Mathematics, D1. Friedr. Vieweg \& Sohn, Braunschweig,
%  1984.

\bibitem[LR79]{LunaRich} D.~Luna and R.~W.~Richardson, \emph{A generalization of
the Chevalley restriction theorem}, Duke Math. J. \textbf{46} (1979), no.~3,
487--496.

\bibitem[Lu73]{Luna} D.~Luna, \emph{Slices \'etales}, Bull.
    Soc. Math. France, M\'emoire \textbf{33} (1973), 81--105.

%\bibitem[Lu75]{Lu75} D.~Luna, \emph{Adh\'erences d'orbite et invariants}, Invent. Math. \textbf{29} (1975), no.~3, %231--238.

\bibitem[Mie07]{Mie07} C.~Miebach,
    \emph{Geometry of invariant subsets in complex semi-simple Lie groups},
Dissertation, Ruhr-Universit\"at Bochum, 2007.

%\bibitem[Mie08]{Mie} C.~Miebach, \emph{Matsuki's double coset decomposition via gradient maps}, J.~Lie Theory, to %appear.

%\bibitem[MFK94]{MFK} D.~Mumford, J.~Fogarty and F.~Kirwan, \emph{Geometric Invariant Theory}, Third edition. %Ergebnisse der Mathematik und ihrer Grenzgebiete 34, Springer-Verlag, Berlin, 1994.

%\bibitem[Ne99]{Ne} C.~Neidhardt, \emph{A convexity theorem for noncommutative gradient flows}, Transform. Groups \textbf{4} (1999), no.~4, 375--404.

\bibitem[Pa61]{Palais} R.~S.~Palais, \emph{On the existence of slices for actions of non-compact Lie
              groups}, Ann. of Math. (2) \textbf{73} (1961), 295--323.


\bibitem[Po98]{Po} D.~Poguntke, \emph{Normalizers and centralizers of reductive subgroups of almost
              connected Lie groups}, J. Lie Theory \textbf{8} (1998), 211--217.

%\bibitem[Ri72]{Richardson} R.~W.~Richardson, \emph{Principal
%    orbit types for algebraic transformation spaces in
%    characteristic zero}, Invent. Math. \textbf{16} (1972),
%    6--14.

\bibitem[RS90]{RS} R.~W.~Richardson and
    P.~J.~Slodowy, \emph{Minimum
    vectors
    for real reductive algebraic groups}, J. London Math. Soc.
    (2) \textbf{42} (1990), no.~3, 409--429.

%\bibitem[Ro72]{Ro} L.~P.~Rothschild, \emph{Orbits in a real reductive Lie algebra}, Trans. Amer. Math. Soc. %\textbf{168} (1972), 403--421.

\bibitem[Sj95]{Sjamaar} R.~Sjamaar, \emph{Holomorphic slices, symplectic reduction and multiplicities of
              representations}, Ann. of Math. (2) \textbf{141} (1995), no.~1, 87--129.

\bibitem[Sn82]{Snow} D.~M.~Snow, \emph{Reductive group actions on Stein spaces}, Math. Ann. \textbf{259} (1982), no.~1, 79--97.

\bibitem[St08]{S08} H.~St\"{o}tzel, \emph{Quotients of real reductive group actions related to orbit type strata}, Dissertation, Ruhr-Universit\"{a}t Bochum, 2008.

\bibitem[St09]{St09} H.~St\"{o}tzel, \emph{Closed orbits of real reductive representations}, in preparation.


\bibitem[Sch80]{Schwarz80} G.~W.~Schwarz, \emph{Lifting smooth homotopies of orbit spaces}, Inst. Hautes \'Etudes Sci. Publ. Math. \textbf{51} (1980), 37--135.

%\bibitem[Wa72]{Warner} G.~Warner, \emph{Harmonic analysis on semi-simple Lie groups. I.}, Die Grund\-lehren der mathematischen Wissenschaften in Einzeldarstellungen, Band 188, Springer-Verlag Berlin Heidelberg New York, 1972. xvi+529 pp.

\end{thebibliography}
\end{document}